\documentclass[twoside,a4paper]{article}
\usepackage{amsmath,amsthm,amssymb,mathrsfs}
\usepackage{enumerate}
\usepackage[dvipdfm]{pict2e}
\usepackage{bbm,cases}
\usepackage{color}
\usepackage{graphicx}

\usepackage{floatflt}

\usepackage{mathptmx}
\usepackage[scaled]{helvet}

\newtheorem{dfn}{Definition}[section]
\newtheorem{thm}[dfn]{Theorem}
\newtheorem{lem}[dfn]{Lemma}

\newtheorem{prop}[dfn]{Proposition}\makeatletter

             \@addtoreset{equation}{section}
\makeatother

\newcommand{\x}{{\mathbbm x}}

\newcommand{\n}{{\bf n}}

\newcommand{\p}{\mathbbm{p}}

\newcommand{\dis}{\displaystyle}
\newcommand{\ve}{\varepsilon}
\newcommand{\Var}{\textrm{Var}}

\newcommand{\J}{\mathbb{J}}

\allowdisplaybreaks[4]
\begin{document}
\title{A remark on the bound for the  free energy of directed polymers in random environment in 1+2 dimension}
\author{Makoto Nakashima \footnote{nakamako@math.tsukuba.ac.jp, Division of Mathematics, Graduate School of Pure and Applied Sciences,University of Tsukuba, 1-1-1 Ten-noudai, Tsukuba-shi, Ibaraki-ken, Japan } }
\date{}
\pagestyle{myheadings}
\markboth{Makoto Nakashima}{Free Energy of 2D DPRE}
\maketitle

\sloppy

\begin{abstract}
We consider the behavior of the quantity $p(\beta)$; the free energy  of directed polymers in random environment in $1+2$ dimension, where $\beta$ is inverse temperature. We know that the free energy is strictly negative  when $\beta$ is not zero. In this paper, we will prove that  $p(\beta)$ is bounded from above by $-\exp\left(	-\frac{c_\ve}{\beta^{2+\ve}}	\right)$ for small $\beta$, where $c_\ve>0$ is a constant depending on $\ve>0$. Also, we will suggest a strategy to get a sharper asymptotics.

\end{abstract}

\vspace{1em}

\vspace{1em}{\bf Key words:}  Directed polymer, Random environment, Free energy.

\vspace{1em}
We denote by  $(\Omega, {\cal F},P )$ a probability space. We denote by $P[X]$ the expectation of random variable $X$ with respect to $P$.  Let $\mathbb{N}=\{0,1,2,\cdots\}$, $\mathbb{N}^*=\{1,2,3,\cdots\}$, and $\mathbb{Z}=\{0,\pm 1,\pm 2,\cdots\}$.  Let $C_{x_1,\cdots,x_p}$ or $C(x_1,\cdots,x_p)$ be a non-random constant which depends only on the parameters $x_1,\cdots,x_p$. For $x=(x_1,\cdots,x_d)\in \mathbb{R}^d$, $|x|$ stands for the $l^1$-norm: $|x|=\sum_{i=1}^d |x_i|$. For $\xi=(\xi_x)_{x\in\mathbb{Z}^d}\in \mathbb{R}^{\mathbb{Z}^d}$, $|\xi|=\sum_{x\in\mathbb{Z}^d}|\xi_x|$.


\section{Introduction}

Directed polymers in random environment was introduced by Huse and Henly to study an interaction for polymer chains and random medium. The model can be represented in terms of polymers and random environment.
\begin{itemize}
\item \textit{The random walk:} $(S,P_S^x)$ is a simple random walk on the $d$-dimensional integer lattice $\mathbb{Z}^d$ with starting $x$. In particular, we write $P_S^0=P_S$. We define $p_n(x,y)=P_S^x(S_n=y)$.

\item \textit{The random environment:} Let $\{\eta(i,x):(i\times x)\in\mathbb{N}^*\times \mathbb{Z}^d\}$ be  i.i.d.\,$\mathbb{R}$-valued random variables defined on the probability space $(\Omega',{\cal G},Q)$. We denote by ${\cal G}_n$ the $\sigma$-field generated by $\{\eta(i,x):1\leq i\leq n, x\in\mathbb{Z}^d\}$. We assume that $Q[\exp(\beta \eta(1,x))]=\exp(\lambda(\beta))<\infty $ for any $\beta\in \mathbb{R}$. We define the shift of environment $\theta_{m}$ by $\theta_{m}\circ \eta(n,x)=\eta(n+m,x)$ for $m,n\in\mathbb{N}$ and $x\in\mathbb{Z}^d$.
\end{itemize}
Then, we can consider Gibbs measure $\mu_n$ on $\mathbb{Z}^d$  under fixed environment $\eta$:\begin{align*}
\mu_n(x)=\frac{1}{W_n}\frac{P_S\left[	\exp(\beta \sum_{i=1}^n\eta(i,S_i))		:S_n=x\right]}{\exp(n\lambda (\beta))},
\end{align*}
where $W_n$ is a partition function of directed polymers in random environment defined by \begin{align}
W_n(\beta)=W_n=P_S\left[	\exp\left(\beta\sum_{i=1}^n\eta(i,S_i)-n\lambda(\beta)\right)	\right].\label{part}
\end{align}
We define the Hamiltonian $H_n(S)$ by \begin{align*}
H_n(S)=\sum_{i=1}^n\eta(i,S_i).
\end{align*}
Then, $W_n$ is a positive martingale on $(\Omega',{\cal G}_n,Q)$ so that there exists the limit  \begin{align*}
W_\infty =\lim_{n\to \infty}W_n ,\ \ Q\text{-a.s.}
\end{align*}
Also, we know that the $0$-$1$ law for $W_\infty$:\begin{align*}
Q(W_\infty=0)\in \{0,1\}.
\end{align*}
We refer to the case $Q(W_\infty=0)=0$ as \textit{weak disorder} and to the other case as \textit{strong disorder}.
Then, it is known the existence of the phase transition, that is  there exist $\beta_-^1(d),\beta_+^1(d)\in [0,\infty]$ (they depend on $d$ and the distribution of environment) such that \begin{align*}
\beta\in (-\beta_-^1(d),\beta_+^1(d))&\Rightarrow \textit{weak disorder}\\
\beta<-\beta_-^1(d),\ \text{or }\beta_+^1(d) <\beta&\Rightarrow \textit{strong disorder.}
\end{align*}

We can investigate  localization/delocalization phenomena in weak/strong disorder resime\cite{Bol,CarHu2,ComShiYos,ImbSpe,ComYos3}.

The following result implies the existence of the free energy $p(\beta)$\cite{ComShiYos,ComYos}:

\begin{prop}There exists the quantity \begin{align*}
p(\beta)=\lim_{n\to \infty} \frac{1}{n}\log W_n,\ \ Q\text{-a.s.}.
\end{align*}
Moreover, $p(\beta)$ is a non-positive, non-random constant and is equal to \begin{align}
\lim_{n\to \infty}\frac{1}{n}Q\left[\log W_n\right]=\sup_n\frac{1}{n}Q\left[\log W_n\right].\label{free}
\end{align}
\end{prop} 
We call the quantity $p(\beta)$ the free energy.

Also, the phase transition occurs on $p(\beta)$: there exist $\beta_-^2(d),\beta_+^2(d)\in[0,\infty]$ such that \begin{align*}
\beta\in (-\beta_-^2(d),\beta_+^2(d))&\Rightarrow p(\beta)=0\\
\beta<-\beta_-^2(d),\ \text{or }\beta_+^2(d)<\beta&\Rightarrow p(\beta)<0.
\end{align*}
We refer to the latter case as the \textit{very strong disorder}.

In particular, when $d=1$ or $2$, \begin{align}
\beta_-^1(d)=\beta_+^1(d)=\beta_-^2(d)=\beta_+^2(d)=0\label{vss}
\end{align}
  and when $d\geq 3$, \begin{align*}
  -\beta_-^2(d)\leq -\beta_-^1(d)<0<\beta_+^1(d)\leq \beta_+^2(d)\end{align*}
   \cite{CarHu,ComShiYos,ComVar,Lac}.
In \cite{Lac}, it is proved that if  $|\beta|$ is small enough, then \begin{align}
-\frac{1}{c}\beta^4(1+(\log \beta)^2)&\leq p(\beta)\leq c\beta^4,\ \ \text{for $d=1$}\label{d1}
\intertext{and}
-\exp\left(-\frac{1}{c\beta^2}\right)&\leq p(\beta)\leq -\exp\left(	-\frac{c}{\beta^{4}}\right),\ \ \text{for $d=2$,}\label{d2}
\end{align}
 where $c$ is a positive constant.
 
 We remark that when environment $\eta$ is Gaussian or infinitely divisible random variables,  $\beta^4$ is crucial for $d=1$ \cite{Lac,Wat}. 
 
There are some results on  asymptotics of free energy for several models related to  directed polymers in random environment. We know the explicit formula of free energy for directed polymers in Brownian environment\cite{MorOCo},  asymptotics of free energy at  low temperature for Brownian directed polymers in Poissonian medium\cite{ComYos2}, and  at high temperature for Brownian directed polymers in Gaussian environment with $d=1$ and $d\geq 1$ with long range correlation\cite{Lac2}.

 In the present paper, we will show that the lower bound in (\ref{d2}) is ``crucial" for $d=2$.  
 \begin{thm}\label{main}
 When $d=2$, for any $\ve>0$ there exist $c_\ve>0$  such that for  $|\beta|\leq \beta_0$\begin{align*}
 p(\beta)\leq -\exp\left(-\frac{c_\ve}{\beta^{2+\ve}}\right).
 \end{align*}  
 \end{thm}

With (\ref{d2}), we may conjecture that as $|\beta|\to 0$ \begin{align}
- \log |p(\beta)|\asymp \frac{1}{\beta^2}. \label{logasym}
\end{align}
Also, we will give a much sharper  conjecture in section \ref{4}.

The proof  is a modification of the proof for the directed polymers in random environment by Lacoin \cite{Lac} and for pinning model  by Giacomin, Lacoin and Toninelli\cite{GiaLacTon}.


\section{Preliminaly}
Before giving the proof, we will give an intuitive reason for (\ref{logasym}) and some calculations which we will use later.
\begin{lem}\label{square}
Let $\beta_N=\frac{\beta}{\sqrt{\log N}}$ for $\beta\in \mathbb{R}$ such that $0<|\beta|<\sqrt{\frac{\pi}{\lambda''(0)}}$. Then, we have that \begin{align*}
0<\sup_{N}Q\left[W_N^2(\beta_N)\right]-1<\infty.
\end{align*}
Moreover, if $|\beta|>\sqrt{\frac{\pi}{\lambda''(0)}}$, then we have that \begin{align*}
\sup_{N}\frac{1}{N}\left(Q\left[W_N^2(\beta_N)\right]-1\right)=\infty.
\end{align*}
\end{lem}
\begin{proof}
We obtain by simple calculation  that 
\begin{align*}
Q\left[W_N^2(\beta_N)\right]&=P_{S,S'}\left[	\exp\left(\left(\lambda(2\beta_N)-2\lambda(\beta_N)\right)\sharp\{0<i\leq N:S_i=S_i'\}\right)	\right],
\end{align*}
where $P_{S,S'}$ is the probability measure of two independent simple random walks. Also, it is represented as 
\begin{align*}
&P_{S,S'}\left[	\prod_{i=1}^N\left(	1+1\{S_i=S_i'\}(\exp(\lambda(2\beta_N)-2\lambda(\beta_N))-1)	\right)	\right]\\
&=\sum_{k=0}^{N}(\exp(\lambda(2\beta_N)-2\lambda(\beta_N))-1)^k\sum_{1\leq i_{1}<\cdots<i_k\leq N}P_{S,S'}(S_{i_j}=S_{i_j}'\ \text{for all }1\leq j\leq k)\\
&=\sum_{k=0}^{N}(\exp(\lambda(2\beta_N)-2\lambda(\beta_N))-1)^k\sum_{1\leq i_{1}<\cdots<i_k\leq N}P_{S}(S_{2i_j}=0'\ \text{for all }1\leq j\leq k)\\
&=\sum_{k=0}^N(\exp(\lambda(2\beta_N)-2\lambda(\beta_N))-1)^k\sum_{1\leq i_{1}<\cdots<i_k\leq N}\prod_{j=1}^kP_S(S_{2i_j-2i_{j-1}}=0).
\end{align*} 
When $N$ is large enough, $\Lambda_{1,N}(\beta)=\exp(\lambda(2\beta_N)-2\lambda(\beta_N))-1=\frac{\lambda ''(0)\beta^2}{\log N}(1+o(1))$.
The local limit theorem implies  that \begin{align*}
Q\left[W_N^2(\beta_N)\right]&\leq \sum_{k=0}^N\left(	\frac{\lambda''(0)\beta^2}{\log N}(1+o(1))	\right)^k\left(\sum_{i=1}^N\frac{1}{i\pi}\right)^k\\
&\leq \frac{1}{1-\frac{\lambda''(0)\beta^2}{\pi}}(1+o(1)).
\end{align*}
Also, we have that \begin{align*}
Q\left[W_N^2(\beta_N)\right]&\geq 1+\frac{\lambda''(0)\beta^2}{\log N}\sum_{i=1}^NP_S(S_{2i}=0)\\
&\geq 1+c\lambda''(0)\beta^2,
\end{align*}
for some $c>0$.

On the other hand, we have for $|\beta|>\sqrt{\frac{\pi}{\lambda''(0)}}$ that \begin{align}
Q\left[W_N^2(\beta_N)\right]-1&=\sum_{k=1}^N\Lambda_{1,N}^k\sum_{1\leq i_{1}<\cdots<i_k\leq N}\prod_{j=1}^kP_S(S_{2i_j-2i_{j-1}}=0)\notag\\
&\geq \Lambda_{1,N}(\beta)^{(\log N)^2}\sum_{1\leq i_{1}<\cdots<i_{(\log N)^2}\leq N}\prod_{j=1}^{(\log N)^2}P_S(S_{2i_j-2i_{j-1}}=0)\notag\\
&\geq \Lambda_{1,N}(\beta)^{(\log N)^2}\left(\sum_{i=1}^{\frac{N}{2(\log N)^2}}P_S(S_{2i}=0)\right)^{(\log N)^2}.\label{lower}
\end{align}

As $N\to \infty$, we have that \begin{align*}
\Lambda_{1,N}(\beta)\left(\sum_{i=1}^{\frac{N}{2(\log N)^2}}P_S(S_{2i}=0)\right)\to \frac{\lambda''(0)\beta^2}{\pi}>1.
\end{align*}
Thus, the right hand side of (\ref{lower}) is bounded from below by $c^{(\log N)^2}$ for some $c>1$ and we complete the proof.

\end{proof}

Thus, we know that the subsequence of $W_N(\beta_N)$ weakly converges to a random variable $W$ and $Q[W]=1$. Then, we find that \begin{align*} 
Q[\log W_N(\beta_N)]``\approx" Q[\log W]<\log Q[ W]=0.
\end{align*}
Therefore, we can conjecture that as $\beta'=\beta_N=\frac{c}{\sqrt{\log N}}\to 0 $ \begin{align*}
\lim_{N\to \infty}\exp\left(\frac{c^2}{\beta'^2}\right)p(\beta')=\lim_{N\to \infty}Q[\log W_N(\beta')]``=" Q[\log W]\leq-{c'}.
\end{align*}
Of course, this approximation is non-rigorous and meaningless. But, we could feel the flavor of the reason for our conjecture of the asymptotics of $p(\beta)$.

\section{Proof of Theorem \ref{main}}

We may assume $\beta>0$.

Now, we consider the partition function $W_{nN}(\beta)$. Let $B_y^{N}$ ($y=(y_1,y_2)\in \mathbb{Z}^2$) be a square in $\mathbb{Z}^2$ defined by\begin{align*}
B_y^N =\{x\in\mathbb{Z}^2:x_i\in [(2y_i-1)\lfloor \sqrt{N}\rfloor +y_i,(2y_i+1)\lfloor \sqrt{N}\rfloor+y_i],i=1,2 \}.
\end{align*}
Then, $B_y^N$ are disjoint and cover $\mathbb{Z}^2$.  


We have by Jensen's inequality that for $\theta\in (0,1)$\begin{align}
\frac{1}{nN}Q\left[\log W_{nN}(\beta)\right]=\frac{1}{nN\theta}Q\left[\log W_{nN}(\beta)^\theta\right]\leq \frac{1}{nN\theta}\log Q\left[W_{nN}(\beta)^\theta\right].\label{anupper}
\end{align}

We will  show that  for some $\beta=Ch(N) $, there exists a $c>0$ such that \begin{align}
Q\left[W_{nN}(\beta)^\theta\right]\leq O(e^{-cn}),\label{expdecay}
\end{align}
where $h(N)\to 0$ as $N\to \infty$.
Taking $n\to\infty$ in (\ref{anupper}), we have that \begin{align*}
p(\beta)\leq -\frac{c}{N\theta}=-\frac{c}{\theta h^{-1}(\beta/C)}.
\end{align*}
In this paper, we take $h(N)=(\log N)^{-\frac{q-1}{2q}}$ for $q\geq 2$, that is $h^{-1}(\beta)=\exp\left(\beta^{-\frac{2q}{q-1}}\right)$. We devote the rest of this section to prove (\ref{expdecay}).

\vspace{1em}

\underline{\textit{Coarse graining}}

We decompose the partition function $Q\left[W_{nN}^\theta(\beta)\right]$ for $n,N\in\mathbb{N}$ into the box.  We have that 
\begin{align*}
W_{nN}=\sum_Z\hat{W}_{z_1,\cdots,z_n},
\end{align*}
where for $Z=(z_1,\cdots,z_n)\in \left(\mathbb{Z}^2\right)^n$, \begin{align*}
\hat{W}_{z_1,\cdots,z_n}=\hat{W}_Z&=P_S\left[\exp\left(\beta H_{nN}(S)-nN\lambda(\beta)\right):		S_{jN}\in {B}_{z_j}^{N},\ j=1,\cdots, n	\right]\\
&=P_S\left[\prod_{i=1}^{nN}	e_{i,S_i}(\beta):S_{jN}\in {B}_{z_j}^{N},\ j=1,\cdots, n\right],
\end{align*}
where $e_{i,x}(\beta)=\exp(\beta \eta(i,x)-\lambda(\beta))$.

Then, we have  that \begin{align}
Q\left[W_{nN}^\theta\right]\leq \sum_{Z}Q\left[	 \hat{W}^\theta_{z_1,\cdots,z_n}	\right],\label{decom}
\end{align}
where we have used the inequality $\dis \left(\sum_{i=1}^\n a_i\right)^\theta \leq \sum_{i=1}^na_i^\theta$ for $a_i\geq 0$ and $\theta\in (0,1)$.
Therefore, it is sufficient to estimate $Q\left[\hat{W}^\theta_Z\right]$. The technique is the same as the one in \cite{Ber,Lac}, change of measure.

For $Z=(z_1,\cdots,z_m)\in\left(\mathbb{Z}^2\right)^m$, we introduce new random variables: For any $q\geq 2$\begin{align*} 
A_{\ell}^{q,N}&=\frac{\sqrt{\log N}}{N}\sum_{x\in B_{z_\ell}^{N}}\left(P^x_S\left[\sum_{1\leq j_i<\cdots<j_q\leq N}\prod_{i=1}^q\Big(\exp\big(	\gamma_N\eta({j_i}+\ell N,S_{j_i})-\lambda(\gamma _N)	\big)-1\Big)\right]\right).
\end{align*}
Here, $\gamma_N=\frac{\gamma}{\sqrt{\log N}}$ for $0<\gamma<\sqrt{\frac{\pi}{\lambda''(0)}}$.


We will look at some properties of $A_{\ell}^{q,N}$.

\begin{lem}\label{corre}
\begin{align*}
Q\left[	\left(A_{\ell}^{q,N}\right)^2	\right]\leq {C^1_{\gamma,\lambda,q}}.
\end{align*}

\end{lem}

\begin{proof}
 It is easy to see that 
 \begin{align*}
&Q\left[	(A_{\ell}^{q,N})^2\right]=\frac{\log N}{N^2}\left(\sum_{x,x'\in B_{z_\ell}^N}\sum_{1\leq j_i<\cdots<j_q\leq N}\Lambda_{1,N}(\gamma)^qP_{S,S'}^{x,x'}\left[S_{j_i}=S_{j_i}',\ i=1,\cdots, q\right]\right).
\end{align*}
Thus, we have that \begin{align*}
Q\left[	(A_{\ell}^{q,N})^2\right]&\leq \frac{\log N}{N^2}\left(\frac{\lambda''(0)\gamma^2(1+o(1))}{\log N}\right)^q\sum_{x,x' \in B_{z_\ell}^{N}}\sum_{i=1}^NP_{S,S'}^{x,x'}\left(	S_{i}=S_{i}'\right)\left(\sum_{i=1}^NP_{S,S'}^{0,0}(S_i=S_i')\right)^{q-1}.
\end{align*}
Since it holds that \begin{align*}
\frac{1}{N^2}\sum_{x,x' \in B_{z_\ell}^{N}}\sum_{i=1}^NP_{S,S'}^{x,x'}\left(	S_{i}=S_{i}'\right)
\leq \frac{1}{N^2}\sum_{i=1}^N \sum_{x\in B_{z_\ell}^{N}}1\leq 1,
\end{align*}
local limit theorem implies the statement.

\end{proof}

Let $K$ be a large constant. We define the function $f_K$ on $\mathbb{R}$ by \begin{align*}
f_K(x)=-K1\{x>\exp(K^2)\}.
\end{align*}

We set the function of the environment by \begin{align}
g_Z(\eta,n,q,N)=\exp\left(		\sum_{k=0}^{n-1}f_K\left(A_{k}^{q,N}\right)		\right).
\end{align}
H\"older's inequality impies that \begin{align}
Q\left[		\hat{W}_Z^\theta	\right]&=Q\left[	(g_Z(\eta,n,q,N))^{-\theta} (g_Z(\eta,n,q,N)\hat{W}_Z)^\theta	\right]\notag\\
&\leq Q\left[g_Z(\eta,n,q,N)^{-\frac{\theta}{1-\theta}}\right]^{1-\theta}Q\left[g_Z(\eta,n,q,N)\hat{W}_Z\right]^\theta.\label{holder}
\end{align}
Then, the first term can be expressed as \begin{align*}
Q\left[g_Z(\eta,n,q,N)^{-\frac{\theta}{1-\theta}}\right]=\prod_{k=0}^{n-1}Q\left[	\exp\left(-\frac{\theta}{1-\theta}f_K\left(A_{k}^{q,N}\right)\right)	\right].
\end{align*}


Lemma \ref{corre} yields that \begin{align*}
Q\left[A_{k}^{q,N}>\exp(K^2)\right]\leq C_{\gamma,\lambda,q}^1\exp(-2K^2)
\end{align*}
and hence \begin{align}
Q\left[g_Z(\eta,n,q,N)^{-\frac{\theta}{1-\theta}}\right]&=\prod_{k=0}^{n-1}Q\left[	\exp\left(-\frac{\theta}{1-\theta}f_K\left(A_{k}^{q,N}\right)\right)	\right]\notag\\
&\leq \left(1+	C_{\gamma,\lambda,q}^1\exp\left(\frac{\theta}{1-\theta}K-2K^2		\right)		\right)^{n}\leq 2^{n}.\label{2m}
\end{align}
for large $K$.

Thus, it is enough to estimate \begin{align*}
&Q\left[g_Z(\eta,n,q,N)\hat{W}_Z\right]\\
&=Q\left[P_S\left[g_Z(\eta,n,q,N)\prod_{i=1}^{nN}e_{i,,S_i}(\beta):  	S_{jN}\in {B}_{z_j}^{N}, \ j=1,\cdots,n	\right]\right]. 
\end{align*}
For fixed trajectory of the random walk $S$, we define a new measure of environment ${Q}_S$ by \begin{align*}
\frac{dQ_S}{dQ}=\prod_{i=1}^{nN}e_{i,S_i}(\beta).
\end{align*}
Then, we have by Fubini's theorem that \begin{align}
&Q\left[P_S\left[g_Z(\eta,n,q,N)\hat{W}_Z\right]\right]=P_S\left[Q_S\left[	g_Z(\eta,n,q,N)		\right]:S_{jN}\in {B}_{z_j}^{N}, \ j=1,\cdots,n	\right]\notag\\
&\leq \prod_{j=1}^n \max_{x\in {B}_{z_{j-1}}^{N}}P_S^x\left[Q_S\left[	\exp\left(f_K\left(A_{j}^{q,N}\right)\right)		\right]:S_{N}\in {B}_{z_j}^{N}\right].\label{max}
\end{align}

Combining (\ref{decom}), (\ref{holder}), (\ref{2m}), and (\ref{max}), \begin{align*}
Q\left[W_{nN}^\theta\right]\leq 2^{m(1-\theta)}\left(\sum_{z\in\mathbb{Z}^2} \max_{x\in {B}_{0}^{N}}P_S^x\left[Q_S\left[	\exp\left(f_K\left(A_{0}^{q,N}\right)\right)		\right]:S_{N}\in {B}_{z}^{N}\right]^\theta\right)^m.
\end{align*}
It is enough to show that \begin{align*}
\sum_{z\in\mathbb{Z}^2} \max_{x\in {B}_{0}^{N}}P_S^x\left[Q_S\left[	\exp\left(f_K\left(A_{0}^{q,N}\right)\right)		\right]:S_{N}\in {B}_{z}^{N}\right]^\theta
\end{align*}
is small.

Since $P_S\left(S_{N}\in {B}_z^{N}\right)$ decays as  $e^{-\frac{|z|^2}{4}}$ as $|z|\to \infty$, for any $\ve>0$ there exists $R>0$ such that \begin{align*}
&\sum_{|z|>R} \max_{x\in {B}_0^{N}}P_S^x\left[Q_S\left[\exp\left(f_K\left(A_0^{q,N}\right)\right)\right]:S_{N}\in {B}_{z}^{N}\right]^\theta\\
&\leq \sum_{|z|>R} \max_{x\in {B}_0^{N}}P_S^x\left[S_{N}\in {B}_{z}^{N}\right]^\theta<\ve.
\end{align*}

In the rest of this section, we will show that \begin{align*}
&\sum_{|z|\leq R}\max_{x\in {B}_0^{N}}P_S^x\left[Q_S\left[\exp\left(f_K\left(A_0^{q,N} \right)\right)\right]:S_{N}\in {B}^{N}_{z}\right]^\theta\\
&\leq 4R^2 \max_{x\in {B}_0^{N}}P_S^x\left[Q_S\left[\exp\left(f_K\left(A_0^{q,N}\right)\right)\right]\right]^\theta
\end{align*}
is small by taking $N\geq \left\lfloor  \exp\left(\left(\frac{C_1}{\beta}\right)^{\frac{2q}{q-1}}\right) \right\rfloor  $ for  $C_1$ large enough. 

\begin{prop}\label{moment}
For any $\ve>0$, there exists $\delta>0$ such that \begin{align*}
\min_{x\in B_0^N}P_{S}^x\left(	X>\delta \left({\log N}	\right)^{q-1}	\right)>1-\ve,
\end{align*}
where \begin{align*}
X=\frac{1}{N}\sum_{y\in B_{0}^N}\sum_{0=j_0<\cdots <j_q\leq N}\prod_{i=0}^{q-1}p_{j_{i+1}-j_i}\left(S_{j_i},S_{j_{i+1}}\right)
\end{align*}
\end{prop}

\begin{proof}

Let \begin{align*}
L&=\frac{1}{N}\sum_{1\leq j_1<\cdots<j_q\leq N}\left(\prod_{i=1}^{q-1}\frac{1\{|S_{j_i+1}-S_{j_{i}}|\leq C_2\sqrt{j_{i+1}-j_{i}}\}}{j_{i+1}-j_{i}}\right)1\{|S_{j_1}-x|\leq C_2\sqrt{j_1}\}
\intertext{and}
D_N^q&=\frac{1}{N}\sum_{1\leq j_1<\cdots<j_q\leq N}\prod_{i=1}^{q-1}\frac{1}{|j_{i+1}-j_{i}|}.
\end{align*}

We take $C_2$ large enough  such that \begin{align*}
P^x_S[L]>\left(1-\frac{\ve}{4R^2}\right)D_N^q.
\end{align*}
Since $0\leq L\leq D^q_N$ almost surely, we have \begin{align}
P^x_S\left(L\geq \frac{D^q_N}{2}\right)\geq 1-\frac{\ve}{2R^2}. \label{Xd}
\end{align}
We remark that \begin{align*}
D_N^q\geq c(\log N)^{q-1}.
\end{align*}

Local limit theorem implies that \begin{align*}
X\geq C_{C_2,q}L.
\end{align*}
where we have used the fact that on the event $\{|S_j-x|\leq C_2\sqrt{j}\}$, there exists $\delta'>0$ such that \begin{align*}
\sum_{y\in B_0^N}p(y,S_j)\geq \delta'\ \ \text{for }1\leq j\leq N .
\end{align*}

\end{proof}

Now, we estimate $A_{0}^{q,N}$ under $Q_S$.

It is easy to see that \begin{align*}
Q_{S}\left[	A_0^{q,N}		\right]=\sqrt{\log N}\left(\exp\left(	\lambda(\gamma_N+\beta)-\lambda(\gamma_N)-\lambda(\beta)	\right)-1\right)^qX
\end{align*}
When $N= \left\lfloor  \exp\left(\left(\frac{C_1}{\beta}\right)^{\frac{2q}{q-1}}\right) \right\rfloor  $, $\beta=\frac{C_{C_1}}{\left(\log N\right)^{\frac{q-1}{2q}}}$ and \begin{align*}
\exp(\lambda(\gamma_N+\beta)-\lambda(\gamma_N)-\lambda(\beta))-1\asymp \gamma_N\beta.
\end{align*}
Thus, we have that \begin{align}
\sqrt{\log N}\left(\exp\left(	\lambda(\gamma_N+\beta)-\lambda(\gamma_N)-\lambda(\beta)	\right)-1\right)^qD_N^q\geq  {C_{C_1,\lambda,\gamma,q}}.\label{1}
\end{align}
We remark that $C_{C_1,\lambda,\gamma,q}$ tends to infinity as $C_1 \to \infty$. 
Therefore, we have from Proposition \ref{moment} that\begin{align*}
\max_{x\in B_0^N}P_S^x\left(	Q_S[A_0^{q,N}]<\delta C_{C_2,q}C_{C_1,\lambda,\gamma,q}		\right)<\frac{\ve}{R^2}.
\end{align*}

We will show that the variance $\Var_{Q_S}\left[		A_0^{q,N}	\right]$ is small enough.

We rewrite $A_{0}^{q,N}$ by \begin{align*}
A_0^{q,N}&=\frac{\sqrt{\log N}}{N}\sum_{x'\in B_0^N}P_{S'}^{x'} \left[		\sum_{1\leq j_1<\cdots<j_q\leq N}\prod_{i=1}^q	\left(	e_{j_i,S'_{j_i}}(\gamma_N)-1	\right)	\right]\\
&=	\sum_{\mathbb{J}^q}V_{\mathbb{J}^q}\prod_{i=1}^q\left(	e_{j_i,x_{i}}(\gamma_N)-1	\right),	
\end{align*}
where $(S',P_{S'}^{x'})$ is a simple random walk starting at $x'$ independent of $S$,  $\J^q=((j_1,x_1),\cdots,(j_q,x_q))$ such that $1\leq j_1<\cdots<j_q\leq N$, $\x=(x_1,\cdots,x_q)\in \left(\mathbb{Z}^2\right)^q$ and we   define \begin{align*}
V_{\J^q}&=\frac{\sqrt{\log N}}{N}\sum_{x\in B_0^{N}}p_{j_1}(x,x_1)p_{j_2-j_1}(x_1,x_2)\cdots p_{j_q-j_{q-1}}(x_q-x_{q-1})\\
&=\frac{\sqrt{\log N}}{N}\sum_{x\in B_0^{N}}\p(x,\J^q).
\end{align*}
We denote by $S(\ell)$ the set of $\J^\ell$ for $\ell\in \mathbb{N}$.

Then, we have that \begin{align*}
&A_0^{q,N}-Q_S\left[	A_0^{q,N}	\right]\\
&=\sum_{\J^q}V_{\J^q}\left(\prod_{i=1}^q		\Big(	e_{j_i,x_i}(\gamma_N)-Q_S\left[e_{j_i,x_i}(\gamma_N)		\right]+Q_S\left[e_{j_i,x_i}(\gamma_N)\right]-1	\Big)-\prod_{i=1}^qQ_S\left[		e_{j_i,x_i}(\gamma_N)-1	\right]\right).
\end{align*}
We set $Q_S\left[e_{i,y}(\gamma_N)\right]-1=m_{i,y}^{\beta,\gamma}$ and $\hat{e}_{i,y}^{\beta,\gamma}=e_{i,y}-m_{i,y}^{\beta,\gamma}-1 $ and omit the superscript $\beta,\gamma$ for convenience.  We remark that \begin{align}
m_{i,y}=\begin{cases}
0\ \ &\text{if }S_i\not= y\\
\exp(\lambda(\beta+\gamma_N)-\lambda(\beta)-\lambda(\gamma_N))-1&\text{if }S_i=y.
\end{cases}
\label{m}
\end{align}
Especially, if $S_i=y$, then $m_{i,y}\sim \lambda''(0)\gamma_N\beta$ as $N\to \infty$.

For $\J^\ell\in S(\ell)$ and $\J^{q-\ell}\in S(q-\ell)$, we denote by $\J^\ell\J^{q-\ell}\in S(q)$ the concatenation of  $\J^\ell$ and $\J^{q-\ell}$, where if some $1\leq j\leq N$ is an  element of $\J^\ell$ and $\J^{q-\ell}$, then we don't define $\J^\ell\J^{q-\ell}$.

Then, we have that \begin{align*}
A_0^{q,N}-Q_S\left[	A_0^{q,N}	\right]&=\sum_{\ell=1}^q\sum_{\J^\ell\J^{q-\ell}}V_{\J^\ell\J^{q-\ell}}\left(\prod_{(j,y)\in \J^{\ell}}\hat{e}_{j,y}	\prod_{(j,y)\in \J^{q-\ell}}	m_{j,y}	\right).
\end{align*}

It yields that \begin{align}
&\Var_{Q_S}\left[	A_0^{q,N}	\right]\notag\\
&\leq q\sum_{\ell=1}^q\sum_{\J^\ell}\sum_{\J_1^{q-\ell},\J^{q-\ell}_2}V_{\J^\ell\J_1^{q-\ell}}V_{\J^\ell\J_2^{q-\ell}}\prod_{(j,y)\in \J_1^{q-\ell}}m_{j,y}\prod_{(j,y)\in \J_2^{q-\ell}}m_{j,y}\prod_{(j,y)\in \J^\ell}Q_{S}\left[\hat{e}_{j,y}^2\right].		\label{var}
\end{align}
We remark that \begin{align*}
Q_S[\hat{e}_{j,y}^2]=\begin{cases}
\exp\left(\lambda(2\gamma_N+\beta)-2\lambda(\gamma_N)-\lambda(\beta)\right)-Q_S[\hat{e}_{j,y}]^2		&\text{if }S_j=y\\
\exp\left(	\lambda(2\gamma_N)-2\lambda(\gamma_N)		\right)-1&\text{if }S_j\not=y.
\end{cases}
\end{align*}
In particular, \begin{align}
Q_S[\hat{e}_{j,y}^2]\sim \begin{cases}
2\lambda''(0)\gamma^2_N&\text{if }S_j=y\\
\lambda''(0)\gamma_N^2&\text{if }S_j\not=y.
\end{cases}\label{e2}
\end{align}

Since it is complicated to estimate (\ref{var}), we will look at the special case  for simplicity.
Let $q=3$ and $\ell=2$. Then, $\J^2=\{(j_1,x_1),(j_2,x_2)\}$ and $\J_1^1=\{(j'_1,x_1')\}$, $\J^1_2=\{(j_1'',x_1'')\}$

Let $\J^\ell_S$ be the subset of $\J^\ell$ defined by \begin{align*}
\J^\ell_S=\{(j,y)\in \J^\ell:S_j=y\}.
\end{align*}

Then, we have that \begin{align*}
\sum_{\J^2}\sum_{\J_1^{1},\J^{1}_2}V_{\J^2\J_1^{1}}V_{\J^2\J_2^{1}}m_{j_1',x_1'}m_{j_1'',x_1''}Q_{S}\left[\hat{e}_{j_1,x_1}^2\right]Q_{S}\left[\hat{e}_{j_2,x_2}^2\right].
\end{align*}
The summations over $\J^1_1$ and $\J_2^1$ are restricted to the case $\J^1_1=\{(j_1',S_{j_1'})\}$ and $\J^1_{2}=\{(J_1'',S_{j_1''})\}$ by (\ref{m}).
When we consider the summation over $\J^2$ with $\J^2(S)=\emptyset$, we have from (\ref{e2}) that \begin{align*}
&\sum_{\J^2,\J^2(S)=\emptyset}\sum_{\J_1^{1},\J^{1}_2}V_{\J^2\J_1^{1}}V_{\J^2\J_2^{1}}m_{j_1',x_1'}m_{j_1'',x_1''}Q_{S}\left[\hat{e}_{j_1,x_1}^2\right]Q_{S}\left[\hat{e}_{j_2,x_2}^2\right]\\
&\leq C_{\lambda}^4\gamma_N^6\beta^2\frac{\log N}{N^2}\sum_{x',x''\in B_0^N}\sum_{(j_1,x_1),(j_2,x_2),(j_1',S_{j_1'}),(j_1'',S_{j''_1})}\p(x',\J^2\J^1_1)\p(x'',\J^2\J^1_2)\\
&\leq C_\lambda^4 \gamma_N^6\beta^2 (\log N)^3,
\end{align*}
where we have used local limit theorem in the last line.

When we take summation over $\J^2$ with $|\J^2(S)|=1$, we have that \begin{align*}
&\sum_{\J^2,|\J^2(S)|=1}\sum_{\J_1^{1},\J^{1}_2}V_{\J^2\J_1^{1}}V_{\J^2\J_2^{1}}m_{j_1',x_1'}m_{j_1'',x_1''}Q_{S}\left[\hat{e}_{j_1,x_1}^2\right]Q_{S}\left[\hat{e}_{j_2,x_2}^2\right]\\
&\leq C_\lambda^4 \gamma_N^6\beta^2 \frac{\log N}{N^2}\sum_{x',x''\in B_0^N}\sum_{(j_1,x_1),(j_2,S_{j_2}),(j_1',S_{j_1'}),(j_1'',S_{j''_1})}\p(x',\J^2\J^1_1)\p(x'',\J^2\J^1_2)\\
&+C_\lambda^4 \gamma_N^6\beta^2 \frac{\log N}{N^2}\sum_{x',x''\in B_0^N}\sum_{(j_1,S_{j_1}),(j_2,x_{2}),(j_1',S_{j_1'}),(j_1'',S_{j''_1})}\p(x',\J^2\J^1_1)\p(x'',\J^2\J^1_2)\\
&\leq C_\lambda^4\gamma_N^6\beta^2 (\log N)^2,
\end{align*}
where we have used local limit theorem in the last line. In particular, we should remark that there exists a constant $C$ such that \begin{align*}
\sum_{j}p_{j}(x,z)p_{j+k}(y,z)<C
\end{align*}
for any $k\in \mathbb{N}$ and $x,y,z\in\mathbb{Z}^2$.

The same thing is true for the case of the summation over $\J^2$ with $|\J^2(S)|=2$. Thus, we have that \begin{align*}
\Var_{Q_S}[A^{q,N}_0]\to 0\ \ \text{as }N\to \infty\ \text{i.e.\,}\beta\to 0.
\end{align*}
This is true for any $q\geq 2$.

Thus, we have that after fixing $C_1$ large, we take $N$ large enough (i.e. $\beta$ is taken small) \begin{align*}
\Var_{Q_S}[A_0^{q,N}]\leq 1
\end{align*}

Thus, we have that \begin{align*}
&\max_{x\in B_0^N}P^x_S\left[	Q_S\left[\exp(f_K(A_0^{q,N}))\right]		\right]\\
&\leq  \max_{x\in B_0^N}P^x_S\left(Q_{S}\left[	A_0^{q,N}\right]<\delta C_{C,1,\lambda,\gamma,q}	\right)\\
& +\max_{x\in B_0^N}P_S^{x}\left[	Q_{S}\left(	A_0^{q,N}<\exp(K^2)	\right):Q_S[A_0^{q,N}]\geq \delta C_{C_1,\lambda,\gamma,q}, 	\right]+e^{-K}.
\end{align*}
Therefore, taking $C_1$ large enough, the right hand side will take small sufficiently and we  complete the proof. \begin{flushright}$\square$\end{flushright}

\section{Some remarks}\label{4}
In this section, we consider the improvement of coarse graining and change of measure argument. In particular, we focus on the asymptotics of $p(\beta)$ for $d=2$ and the critical points $\beta^1_\pm(d)$ and $\beta^2_\pm(d)$.

We should remark that the proof is true for any dimensional case until (\ref{holder}) by retaking $B_y^N$ and  $g_Z$ to $d$-dimensional version. We will change $f_K$ and $A^{q,N}$ in the following argument.

\subsection{Remark and conjecture}\label{rem}

In this subsection, we will  give remarks on the proof and a conjecture of the asymptotics of $p(\beta)$.

Roughly speaking, the following is essential for change of measure technique used in the proof : Find some nice random variable $V\in {\cal G}_N$ such that \begin{enumerate}
\item $Q[V]=0$ and  $Q\left[V^2\right]<\infty$.
\item $Q\left[	V\prod_{i=1}^N	e_{i,S_i}\right]$ will be large with high probability (and more nice properties). 
\end{enumerate}

To obtain much sharper upper bound of $p(\beta)$ by using change of measure argument, we should find good $V$ satisfying (1) and (2).

To find such $V$, we consider an extension of $A_0^{q,N}$. One of extensions of $A_0^{q,N}$ is \begin{align}
V^N=\frac{\sqrt{\log N}}{N}\sum_{y\in B_0^N}P_{S}^y\left[		\prod_{i=1}^N\exp\left(	\gamma_N\eta(i,S_i)-\lambda(\gamma_N)	\right)-1	\right].\label{V}
\end{align}
Indeed,  the $q$-th expansion of the right hand side as Lemma \ref{square} is equal to $A_0^{q,N}$. Letting $\gamma$ be a constant with $0<|\gamma|<\sqrt{\frac{\pi}{\lambda''(0)}}$, that is   \begin{align*}
\sup _NQ\left[W_N^2\left(\frac{\gamma}{\sqrt{\log N}}\right)\right]<\infty.
\end{align*} 
Then, we have that \begin{align*}
\sup_NQ\left[	\left(V^N\right)^2		\right]<\infty.
\end{align*}

Actually, we have that \begin{align*}
Q\left[\left(V^N\right)^2\right]&=\frac{\log N}{N^2}\sum_{y',y''\in B_0^N}P_{S',S''}^{y',y''}\left[	\exp\left(	\left(\lambda(2\gamma_N)-2\lambda(\gamma_N)\right)\sharp\{1\leq i\leq N:S'_i=S''_i\}		\right)	-1	\right]\\
&=\frac{\log N}{N^2}\sum_{k=1}^{N}\sum_{y',y''\in B_0^N}\sum_{1\leq j_1<\cdots<j_k\leq N}\Lambda_{1,N}(\gamma)^{k}P_{S',S''}^{y',y''}\left(S'_{j_i}=S_{j_i}'':i=1,\cdots,k\right)\\
&\leq \frac{\log N}{N^2}\sum_{k=1}^N\sum_{y',y''\in B_0^N}\Lambda_{1,N}(\gamma)^{k}\sum_{1\leq j_1\leq N}P_{S',S''}^{y',y''}(S'_{j_1}=S_{j_1}'')\left(\sum_{1\leq j\leq N}P_{S',S''}^{0,0}(S'_j=S_j'')\right)^{k-1}\\
&\leq \frac{1}{N^2}\sum_{k=1}^N\sum_{y',y''\in B_0^N}\sum_{1\leq j_1\leq N}P_{S',S''}^{y',y''}(S'_{j_1}=S_{j_1}'')\pi \left(	\frac{\lambda''(0)\gamma^2(1+o(1))}{\pi}	\right)^{k}\\
&\leq C_{\lambda,\gamma}<\infty.
\end{align*}

On the other hand, we take $\beta$ with $\gamma \beta >\frac{\pi}{\lambda''(0)}$. 
Then, we have that \begin{align*}
\sup_{N\geq 1}\min_{x\in B_0^N}P_S^x\left[	Q_S\left[V^N\right]	\right]=\infty.
\end{align*}
Indeed, we have that \begin{align*}
&P_S^x\left[Q_S\left[V^N\right]\right]\\
&\geq P_S^x\left[Q\left[	\frac{\sqrt{\log N}}{N}P_{S'}^x\left[\left(\prod_{i=1}^N\exp\left(\gamma_N\eta(i,S_i')-\lambda(\gamma_N)\right)-1\right)\prod_{i=1}^N\exp(\beta_N\eta(i,S_i)-\lambda(\beta_N))		\right]	\right]\right]\\
&=\frac{\log N}{N}P_{S,S'}^{x,x}\left[	\Lambda_{2}(\gamma_N,\beta_N)^{I(S,S')}-1	\right],
\end{align*}
where $\Lambda_2(\gamma_N,\beta_N)=\exp\left(	\lambda(\gamma_N+\beta_N)-\lambda(\gamma_N)-\lambda(\beta_N)	\right)-1$ and $I(S,S')=\sharp\{1\leq i\leq N:S_i=S_i'\}$. We know that $\left(\Lambda_2(\gamma_N,\beta_N)-1\right)\log N \to \gamma\beta \lambda''(0)$.

We find the last term tends to infinity by the same argument as the proof of Lemma \ref{square}. 

Thus, $V^N$ defined in (\ref{V}) may be  one of the nice choice.  Also, since  we chose $0<\gamma<\sqrt{\frac{\pi}{\lambda''(0)} }$  and $\beta \gamma>\frac{\pi}{\lambda''(0)}$ arbitrarily, $\beta>\sqrt{\frac{\pi}{\lambda''(0)} }$. Thus, we have that for $\beta'=\frac{\beta}{\sqrt{\log N}}$ \begin{align*}
p(\beta')\leq -\frac{C}{N}=-C\exp\left(-\frac{\beta^2}{(\beta')^2}\right)
\end{align*}

Also, the lower bound of $p(\beta')$ is improved to \begin{align*}
-C\exp\left(-\frac{c}{(\beta')^2}\right)
\end{align*}by modifying the proof by Lacoin \cite{Lac}, where $0<c<\frac{\pi}{\lambda''(0)}$. 

Thus, we have a conjecture of logarithmic asymptotics of $|p(\beta)|$:\begin{align*}
\log |p(\beta)|\sim -\frac{\pi}{\lambda''(0)\beta^2}\ \ \text{as }\beta\searrow 0.
\end{align*}

At least, we may prove \begin{align*}
\log |p(\beta)|\asymp -\beta^{-2}\ \ \text{as }\beta\searrow 0
\end{align*}
by showing that \begin{align*}
\liminf_{N\to\infty}\frac{1}{N}P_{\tilde{S}}\left[\exp\left(	\frac{\Gamma\sharp\{1\leq i\leq N:S_i=\tilde{S}_i\}}{\log N}		\right)\right]=\infty\ \ \text{$P_S$-a.s.}
\end{align*}\ for $\Gamma>0$ large enough.

\subsection{$\beta^1_\pm(d)=\beta^2_\pm(d)$ ?}

In this subsection, we consider the problem $\beta^1_\pm(d)=\beta_\pm^2(d)$ or not. We know that $\beta^1_\pm(d)=\beta_\pm^2(d)$ for $d=1,2$ (see\ (\ref{vss})). 

We don't have any conjecture for $d\geq 3$ but the following argument is plausible for us to believe $\beta^1_\pm(d)=\beta^2_\pm(d)$ for $d\geq 3$.

We will apply an \textit{extension }of the proof considered in subsection \ref{rem} into the higher dimensional case.

The coarse graining argument doe hold for any dimensional case as mentioned in the top of this section. So we will modify the change of measure method.  
We  take $V^N$ as \begin{align*}
V^N&=\frac{1}{(2\sqrt{N})^{d}}\sum_{x\in [-\sqrt{N},\sqrt{N}]^d}P_{S'}^x\left[	\prod_{i=1}^N\exp\left(	\gamma(\eta(i,S'_i))-\lambda(\gamma)	\right)		\right]\\
&=\frac{1}{(2\sqrt{N})^{d}}\sum_{x\in [-\sqrt{N},\sqrt{N}]^d}W_N^x(\gamma)
\end{align*}
where $W_N^x(\gamma)=\dis P_S^x\left[\exp\left(	\gamma(\eta(i,S_i))-\lambda(\gamma)	\right)\right] $ and  $0<\gamma<\beta^1_+(d)$ (we will take $\gamma$ close enough to $\beta^1_+(d)$ later). 

Since $W_N^x(\gamma)$ is the partition function in the weak disorder phase, we have that for any $x\in\mathbb{Z}^d$, $W^x_N$ converges to some random variable $W_\infty^x(\gamma)$ in $L^1$. Also, $W^x_\infty$ is shift invariant. Therefore, we have from Birkhoff's  ergodic theorem that there exists a random variable $W$ such that  \begin{align*}
\frac{1}{(2\sqrt{N})^d}\sum_{y\in [-\sqrt{N},\sqrt{N}]}W_\infty^y\to W\ \ \ Q\text{-a.s.\,and in $L^1$.}
\end{align*} 
In particular, $V^N $ converges to $W$ in $L^1$.
We define $f_{K,M}(x)=-K1\{x>M\}$. Taking $M$ large enough,  we have \begin{align*}
\sup_{N}Q\left[\exp\left(-\frac{\theta}{1-\theta}f_{K,M}(V^N)\right)\right]\leq 2
\end{align*}
as (\ref{2m}). The same argument before Proposition \ref{moment} implies that it is enough to show that $\dis \max_{x\in B_0^N}P_S^x\left[	Q_S\left[\exp\left(f_{K,M}(V^N)\right)\right]		\right]$		is small 	for $\beta>\beta^1_+(d)$.

Let $\beta>\beta^{1}_+(d)$ and take $\gamma$ such that $\lambda(\beta+\gamma)-\lambda(\beta)-\lambda(\gamma)>\lambda(2\beta^1_+(d))-2\lambda(\beta^1_+(d))$. Then, we have that \begin{align*}
\lim_{n\to \infty}\frac{1}{\sqrt{N}^d}P_{S,S'}\left[	\left(\exp(\lambda(\beta+\gamma)-\lambda(\gamma)-\lambda(\beta))	\right)^{\sharp\{1\leq i\leq N:S_i=S_i'\}}|S	\right]=\infty
\end{align*}
a.s.\ since $\lambda(2\beta^1_+(d))-2\lambda(\beta^1_+(d))\geq \log \alpha_*$ \cite{Bir2} and the above conditional  expectation diverges exponentially \cite{Bir1}.
Thus, it found that  the mean of $V^N$ under $Q_S$ diverges exponentially and we may have that  \begin{align}
\exp(f_{K,M}(V^N))=\exp(-K) \text{ with high $Q_S$-probability}\label{qhigh}
\end{align} by taking $N$ large and $\dis \max_{x\in B_0^N}P_S^x\left[	Q_S\left[\exp\left(f_{K,M}(V^N)\right)\right]		\right]$		is small. 

Therefore, we may have that \begin{align*}
\lim_{n\to \infty}\frac{1}{nN}Q[\log W_{nN}(\beta)]\leq -\frac{\ve}{N}<0.
\end{align*} 
Since $\beta>\beta_+^1(d)$ is arbitrary, $\beta^1_+(d)=\beta^2_+(d)$. In particular, if we can prove (\ref{qhigh}) (or similar statements), then we can show $\beta^1_\pm(d)=\beta^2_\pm(d)$.

\vspace{2em}

{\bf Acknowledgement:} The author appreciates Prof.\,Nobuo Yoshida for his careful reading, comments and the fruitful discussion. This research was supported by JSPS Grant-in-Aid for Young Scientists (B) 26800051.

\nocite{*}


\end{document}